\documentclass[12pt]{article}
\usepackage{amsfonts}
\usepackage{amssymb}
\usepackage{amsmath}
\usepackage{amsthm}
\usepackage{graphicx}

\usepackage{pgfplots}

\usepackage[miktex]{gnuplottex}

\newtheorem{theorem}{Theorem}[section]

\newtheorem{lemma}[theorem]{Lemma}

\title{Ultraspherical moments on a set of disjoint intervals }
\author{
Hashem Alsabi \\
Universit\'e des Sciences et Technologies, Lille 1, France\\
hashem.alsabi@gmail.com\\
James Griffin\\
Department of Mathematics\\
American University of Sharjah, UAE \\ 
jgriffin@aus.edu}

\begin{document}
\maketitle

\begin{abstract}
Moment evaluations are important for the study of non-classical orthogonal polynomial systems for which explicit representations are not known. In this paper we compute, in terms of the hypergeometric function, the moments associated with a generalized ultraspherical weight on a collection of intervals with two symmetric gaps. These moments, parametrized by the endpoints of the gaps, are identified as a one parameter deformation between the full range ultraspherical moments and the half range ultraspherical moments. 
\end{abstract}



\noindent
\section{Introduction}
\noindent
The full range Ultraspherical moments are defined by the following integral over the interval $[-1,1]$
\begin{equation} \label{intclassic}
\mu_n := \int_{-1}^1 x^n(1-x^2)^{\mu-1/2} \; dx \qquad n = 0,1,2,3,.. \qquad Re(\mu) > -1/2.
\end{equation}
The moments have the following representation in terms of the Gamma function
\begin{equation} \label{intclassiceval}
\begin{split}
\mu_{2n} &= \frac{\Gamma(n+1/2)\Gamma(\mu+1/2)}{\Gamma(n+\mu+1)} \\
\mu_{2n+1} &= 0 
\end{split}
\end{equation}
The half range Ultraspherical moments are defined over the interval $[0,1]$ as follows
\begin{equation} \label{halfeval}
\mu^H_{n}:= \int_0^1 x^n(1-x^2)^{\mu-1/2} \; dx \qquad n = 0,1,2,3,.. \qquad Re(\mu) > -1/2
\end{equation}
and have the following representation in terms of the Gamma function
\begin{equation} \label{halfeval2}
\mu^H_{n} = \frac{1}{2} \frac{\Gamma(n/2+1/2)\Gamma(\mu+1/2)}{\Gamma(n/2+\mu+1)}
\end{equation}
In this paper we construct a generalization of the classical ultraspherical moment integral (\ref{intclassic}) in terms of an integral over a set of disjoint intervals with two gaps of the same length. The particular configuration of disjoint intervals that we consider can be parametrized by one-parameter and we will show that the moments defined in this paper form a one-parameter deformation between the classical, full range, ultraspherical moments on the interval $[-1,1]$, and the half-range ultrapsherical moments on the interval $[0,1]$.
\vskip 0.2cm
\noindent
The construction we use follows an analagous construction to that used in \cite{Griffin-09} and \cite{Griffin-11} where one gap was considered, and is founded on the relationship between a collection of disjoint intervals and a diophantine polynomial equation known as a generalized Pell equation. The generalized ultraspherical moments form a one parameter generalization of the classical case and we show that the moments can be explicitly evaluated in terms of the hypergoemteric function and this parameter. 
\vskip 0.2cm
\noindent
Explicit evaluation of moments such as these is important in general for further study of the associated orthogonal  polynomials and related quantities, where one does not have an explicit representation for the polynomials or the recurrence coefficients in terms of special functions. See for example, \cite{Huybrechs} and \cite{Orel-Perne} where they make use of the half-range Chebyshev moments in a numerical study applied to the efficient computation of Fourier approximations. See also \cite{MRR} where computations performed with half range Legendre moments are applied to radiative transfer in spherical geometry. 
\vskip 0.1cm
\noindent
This paper is organized as follows. In section $2$ we give a brief summary of the construction of the 1-gap case, studied in \cite{Griffin-09} and \cite{Griffin-11}, together with the explicit representation of the moments. Before moving to the two gap case we evaluate the odd moments again using a different method which is based on the polynomial mapping technique first outlined in \cite{GVA}. This method results in a different representation which involves the classical hypergeometric function. We finish section $2$ by describing in detail the construction of the two gap case. In section $3$ we compute the moment integral in the two gap case explicitly in terms of the classical hypergeometric function, making use of the revised method in section $2$. Finally we relate the limiting case, as the gaps become as large as possible, to the half range ultraspherical moments on the interval $[0,1]$.
\noindent
\section{Construction of the moment integral}
\noindent
\subsection{Revisiting the one gap case}
In this section we briefly review the construction of the moment integral for the case of one gap and use this opportunity to express the odd moments from \cite{Griffin-11} in a more convenient form for the purpose of this paper. 
\vskip 0.2cm
\noindent
The Chebyshev polynomials of the first and second kind, which are orthogonal on the set $[-1,1]$, are well known to have parametric representations in terms of trigonometric functions. These representations derive mainly from the Polynomial Pell equation which is
\begin{equation} \label{Pell}  
P_n^2(x) - (x^2-1)Q_n^2(x) = 1.
\end{equation}
If $P$ and $Q$ are polynomials then it can be shown that the only solutions to the above equation are the Chebyshev polynomials of the first and second kind, $T_n$ and $U_n$ respectively. See \cite{Nathanson} and associated references for more details on the connection between the Pell equation and the Chebyshev polynomials. Akhiezer's treatise on elliptic functions, \cite{Akhiezer}, see also \cite{Akhiezer-2}, introduced a generalization of the Chebyshev polynomials that were associated with the following modified Pell equation
\begin{equation}
(x-\alpha)P_n^2(x) - (x^2-1)(x-\beta)Q_n^2(x) = (x-\gamma_n).
\end{equation}
Akhiezer showed that the only polynomial solutions of this equation are polynomials orthogonal on the interval with one gap :
\begin{equation}
E := [-1,\alpha] \cup [\beta,1] \qquad  -1 < \alpha \leq \beta < 1
\end{equation}
with respect to the weight function
\begin{equation}
w_E(x) = \sqrt{\frac{x-\alpha}{(1-x^2)(x-\beta)}}.
\end{equation}
Such polynomials have an analogous parametric representation in terms of elliptic functions. In \cite{ChLa} the technique of Akhiezer was generalized in order to obtain an explicit representation for generalized polynomials orthogonal on an interval with any number of gaps in terms of abelian functions.
\vskip 0.1cm
\noindent
In \cite{GVA}, a theory of orthogonal polynomials on several intervals with a general weight was developed. Under this theory the gaps could not be arbitrarily chosen, instead they had to be chosen in such a way that the several intervals were each a preimage of $[-1,1]$ under a polynomial mapping of some degree. For example, the set 
\begin{equation}
E_2 = [-1,-b] \cup [b,1] \qquad b \in [0,1)
\end{equation}
can be mapped onto the set $[-1,1]$ using a polynomial of degree $2$. So, generalized Chebyshev polynomials on the set $E_2$ can be studied from the point of view of Akhiezer or from the point of view of the theory of polynomial mappings. In \cite{Griffin-09}, and later in \cite{Griffin-11}, a further connection between the Pell equation and the set $E_2$ was explored. Specifically, the Pell equation on $E_2$ 
\begin{equation} \label{PellE2}
(x+b)P_n^2(x) - (x^2-1)(x-b)Q_n^2(x) = (x- (-1)^{n+1}b),  
\end{equation} 
was normalized to obtain
\begin{equation}
\frac{x+b}{(x- (-1)^{n+1}b)}P_n^2(x) - \frac{(x^2-1)(x-b)}{(x- (-1)^{n+1}b)}Q_n^2(x) = 1.
\end{equation}
Under this normalization the following quantity
\begin{equation}
\sqrt{\frac{x+b}{(x- (-1)^{n+1}b)}}P_n(x) \qquad  x \in E_2
\end{equation}
is a function that is somewhat analogous to the Chebyshev polynomial of degree $n$ in the sense that it tends to $T_n(x)$ as $b \to 0$. In particular for $n=1$ we have
\begin{equation}
\sqrt{\frac{x+b}{x-b}}P_1(x) = \sqrt{\frac{x^2-b^2}{1-b^2}}\frac{x}{|x|} \qquad x \in E_2
\end{equation}
In \cite{Griffin-11} this function was denoted as follows
\begin{equation}
\cos \phi_{E_2} := \sqrt{\frac{x^2-b^2}{1-b^2}}\frac{x}{|x|} \qquad x \in E_2
\end{equation}
to draw the analogy with the classical Chebyshev case where
\begin{equation}
\cos \theta = x \qquad x \in [-1,1].
\end{equation}
In \cite{Griffin-09} and \cite{Griffin-11} the following moment integrals over the set $E_2$ were computed where $n=0,1,2,...$ :
\begin{equation} \label{1-gap-even}
\int_{E_2} (\cos \phi_{E_2})^{2n}(1-\cos^2 \phi_{E_2})^\mu \; w_{E_2}(x) \; dx = \frac{\Gamma(\frac{1}{2})\Gamma(\mu + \frac{1}{2})}{\Gamma(\mu+1)}\frac{\left(\frac{1}{2} \right)_n}{(\mu+ 1)_n} \qquad Re(\mu) > -\frac{1}{2}
\end{equation}
\begin{equation} 
\begin{split} \label{1-gap-odd}
& \int_{E_2} (\cos \phi_{E_2})^{2n+1}(1-\cos^2 \phi_{E_2})^\mu w_{E_2}(x) \; dx \\
& \qquad \qquad \qquad   =\frac{\Gamma(\frac{1}{2})}{\Gamma(n+1+\mu)} \sum_{k=1}^{\infty}(-1)^{k+1}\frac{\Gamma(n+1-\frac{k}{2})\Gamma(\mu+\frac{k}{2})}{\Gamma(1-\frac{k}{2})\Gamma(\frac{1}{2}+\frac{k}{2})}b^k  \qquad Re(\mu) > -\frac{1}{2}
\end{split}
\end{equation}
In (\ref{1-gap-even}) we are using the standard notation
\begin{equation} \label{shiftfac}
(a)_n := a(a+1)...(a+n-1) \qquad (a)_0 := 1
\end{equation}
and in (\ref{1-gap-odd}) the term
\begin{equation*}
\frac{\Gamma(n+1-\frac{k}{2})}{\Gamma(1-\frac{k}{2})}
\end{equation*}
is to be understood as a polynomial and is therefore well defined for even values of $k$. 
Note that as $b \to 0$ the equations (\ref{1-gap-even}) and (\ref{1-gap-odd}) revert to 
\begin{equation}
\int_{-1}^1x^{2n}(1-x^2)^{\mu-1/2}\; dx = \frac{\Gamma(\frac{1}{2})\Gamma(\mu + \frac{1}{2})}{\Gamma(\mu+1)}\frac{\left(\frac{1}{2} \right)_n}{(\mu+ 1)_n} \qquad Re(\mu) > -\frac{1}{2}
\end{equation}
and 
\begin{equation}
\int_{-1}^1x^{2n+1}(1-x^2)^{\mu-1/2}\; dx = 0  \qquad Re(\mu) > -\frac{1}{2}
\end{equation}
which are the classical ultraspherical moments, this time with the $n$-dependance written in terms of the shifted factorial (\ref{shiftfac}). 
\vskip 0.2cm
\noindent
Before proceeding to the two gap case we now compute the odd moments in (\ref{1-gap-odd}) using the method of polynomial mappings to arrive at a representation of the odd moments in terms of the hypergeometric function. We will make use of the following facts that can be easily verified :
\begin{itemize}
\item  For $n=2$, up to a change in sign, the solutions of (\ref{PellE2}) are 
\begin{equation} \label{P2Q2}
P_2(x) = \frac{2}{1-b^2}x^2 - \frac{1+b^2}{1-b^2} \qquad \textrm{and} \qquad Q_2(x) = \frac{2}{1-b^2}x + \frac{2b}{1-b^2}
\end{equation}
\item
\begin{equation}
\{ x \in {\bf R} \; | \; P_2(x) \in [-1,1] \} = E_2.
\end{equation}
\item  
\begin{equation}
\cos^2{\phi_{E_2}} = \frac{P_2+1}{2} \qquad \textrm{specifically} \qquad 
 \cos \phi_{E_2} = \left\{
\begin{array}{cc}
-\sqrt{\frac{P_2+1}{2}} & x \in [-1,-b] \\
\sqrt{\frac{P_2+1}{2}} & x \in [b,1]
\end{array}
\right. 
\end{equation}
\item
\begin{equation}
w_{E_2}(x) = \frac{|Q_2(x)|}{\sqrt{1-P_2^2(x)}}
\end{equation}
\end{itemize}
Making use of the points $3$ and $4$ above, the left hand side of (\ref{1-gap-odd}) and be written as
\begin{equation}
\begin{split}
& -\int_{-1}^{-b} \left(\frac{P_2(x)+1}{2} \right)^{n+1/2} \left(\frac{1-P_2(x)}{2} \right)^\mu  \frac{|Q_2(x)|}{\sqrt{1-P_2^2(x)}}  dx \\
& \qquad + \int_{b}^{1} \left(\frac{P_2(x)+1}{2} \right)^{n+1/2} \left(\frac{1-P_2(x)}{2} \right)^\mu  \frac{|Q_2(x)|}{\sqrt{1-P_2^2(x)}}  dx
\end{split}
\end{equation}
Making the substitution $ z = P_2(x) $ in the above integrals we obtain
\begin{equation}
\int_{-1}^1\left(\frac{z+1}{2} \right)^{n+\frac{1}{2}}\left(\frac{1-z}{2} \right)^{\mu}\left[-\frac{Q_2(P_{1,2}^{-1}(z))}{P_2^{\prime}(P_{1,2}^{-1}(z))}+\frac{Q_2(P_{2,2}^{-1}(z))}{P_2^{\prime}(P_{2,2}^{-1}(z))} \right] \frac{dz}{\sqrt{1-z^2}}
\end{equation}
Here the function $P_{1,2}^{-1}$ is defined to be the inverse of the function $P_2$ where $P_2$ is restricted to the domain $[-1,-b]$. Similarly, the function  $P_{2,2}^{-1}$ is defined to be the inverse of the function $P_2$ where $P_2$ is restricted to the domain $[b,1]$. To simplify the terms in the square brackets we form the following partial fractions decomposition
\begin{equation}
\frac{Q_2(s)}{P_2(s) - z} =  \frac{w_1(z)}{s - P_{1,2}^{-1}(z)} +  \frac{w_2(z)}{s - P_{2,2}^{-1}(z)}  \qquad z \in (-1,1) 
\end{equation}
It follows that  
\begin{equation}
w_i(z) =  \frac{Q_2(P_{i,2}^{-1}(z))}{P_2^{\prime}(P_{i,2}^{-1}(z))} \qquad i=1,2 \qquad  \textrm{and} \qquad w_1(z) + w_2(z) = 1.
\end{equation}
Therefore we rewrite the integral as 
\begin{equation} \label{int1}
\int_{-1}^1\left(\frac{z+1}{2} \right)^{n+\frac{1}{2}}\left(\frac{1-z}{2} \right)^{\mu}\left[1-2\frac{Q_2(P_{1,2}^{-1}(z))}{P_2^{\prime}(P_{1,2}^{-1}(z))} \right] \frac{dz}{\sqrt{1-z^2}}
\end{equation}
We compute the following quantity directly from (\ref{P2Q2})
\begin{equation}
P_{1,2}^{-1}(x) = - \frac{1}{\sqrt{2}} \sqrt{(1+b^2)+x(1-b^2)}
\end{equation}
from which we find that 
\begin{equation}
1-2\frac{Q_2(P_{1,2}^{-1}(z))}{P_2^{\prime}(P_{1,2}^{-1}(z))} = \frac{b\sqrt{2}}{\sqrt{1+b^2}}\frac{1}{\sqrt{1-\frac{b^2-1}{b^2+1}z}}
\end{equation}
Therefore (\ref{int1}) becomes
\begin{equation}
\frac{b\sqrt{2}}{\sqrt{1+b^2}}\frac{1}{2^{n+\frac{1}{2}+\mu}} \int_{-1}^1 (z+1)^{n}(1-z)^{\mu-\frac{1}{2}}\left(1 - \frac{b^2-1}{b^2+1}z \right)^{-\frac{1}{2}} dz
\end{equation}
which through the substitution $ u = \frac{1+z}{2} $ can be written as
\begin{equation}
\int_0^1 u^{n}(1-u)^{\mu-\frac{1}{2}}\left(1-\frac{(b^2-1)}{b^2}u\right)^{-\frac{1}{2}} du = \frac{\Gamma(n+1)\Gamma(\mu+\frac{1}{2})}{\Gamma(\mu+n+\frac{3}{2})}\;_2F_1\left(\frac{1}{2},n+1;\mu+n+\frac{3}{2}; \frac{b^2-1}{b^2}\right)
\end{equation}
Finally, using Pfaff's transformation  \cite{AAR} page 68,  on the hypergeometric function above we obtain
\begin{equation} \label{1-gap-odd-rep2}
\begin{split}
& \int_{E_2} (\cos \phi_{E_2})^{2n+1}(1-\cos^2 \phi_{E_2})^\mu w_{E_2}(x) \; dx \\
& \qquad \qquad \qquad = b \frac{\Gamma(n+1)\Gamma(\mu+\frac{1}{2})}{\Gamma(\mu+n+\frac{3}{2})}\;_2F_1\left(\frac{1}{2},\mu+\frac{1}{2};\mu+n+\frac{3}{2}; 1-b^2\right)
\end{split}
\end{equation}
as an alternative expression for the odd moments. Figure $1$ shows an example of the equivalence of the representations (\ref{1-gap-odd}) and (\ref{1-gap-odd-rep2}) in the case where $\mu=0$. In the graph, curve $A$ shows an example of the representation (\ref{1-gap-odd}) for the first odd moment with $\mu =0$, and curve $B$ shows the same example but using the representation (\ref{1-gap-odd-rep2}). We can see that the curves coincide for $b \in [0,1)$. 
\begin{figure}[h] 
\begin{center}
\center{\includegraphics[width=10cm,height=9cm,angle=0]{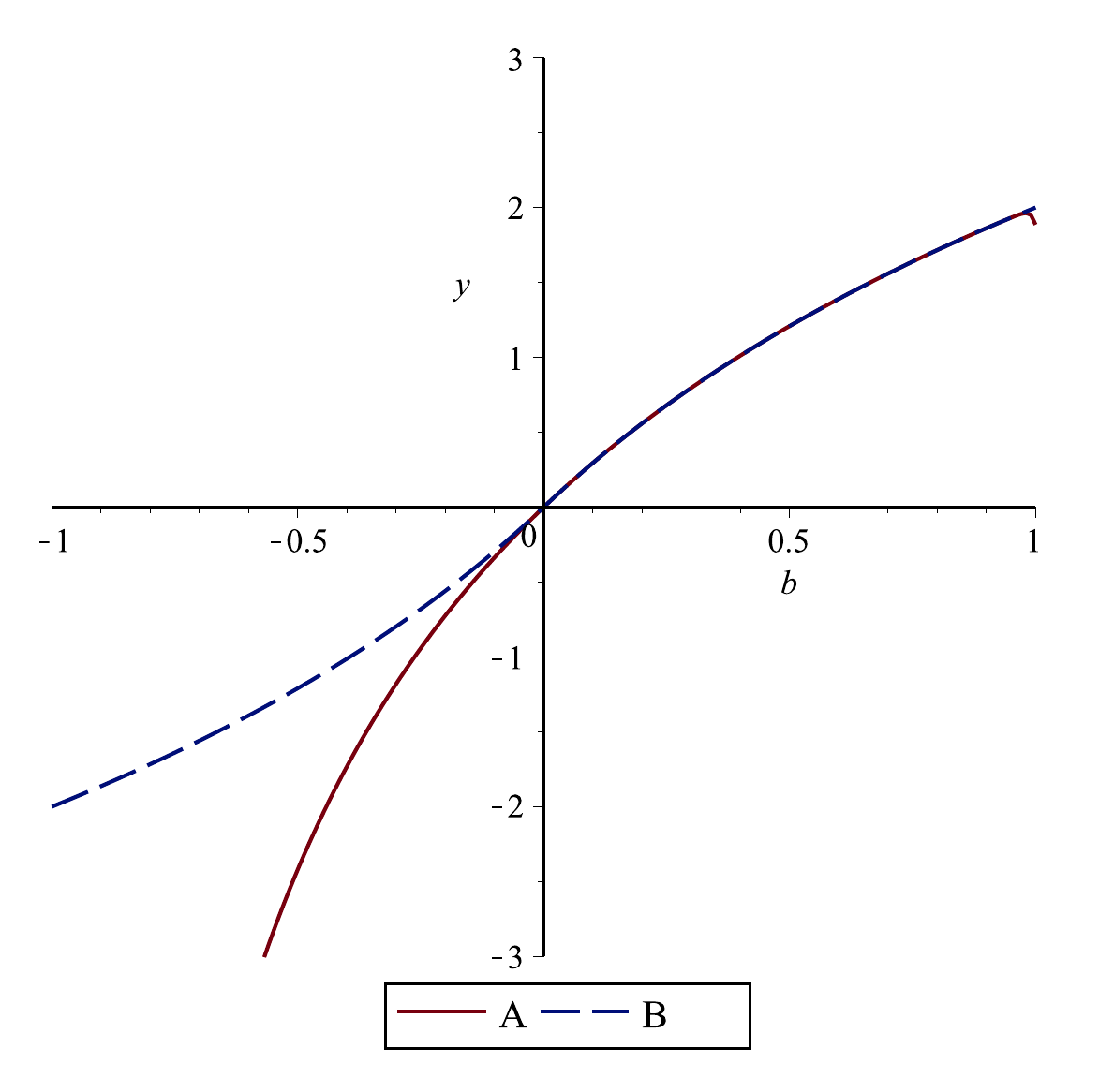}}
\end{center}
\caption{Plot of the first odd moment as a function of $b$ in the case where $\mu = 0$.}
\label{3fns}
\end{figure}

\subsection{Two gap case}
For the remainder of this paper, we carry out a similar study to that described in the previous subsection, this time working with a set with two symmetric gaps. Specifically we work on the following set 
\begin{equation}
E_4 := [-1,-b] \cup [-\sqrt{1-b^2},\sqrt{1-b^2}] \cup [b,1] \qquad b \in \left[\frac{1}{\sqrt{2}}, 1\right)
\end{equation}
This particular set is of interest as it can be thought of as the pre-image of $[-1,1]$ under the following polynomial mapping of degree $4$ (see figure $2$) :
\begin{equation} \label{map}
 P_4(x) = \frac{2x^4-2x^2+b^2(1-b^2)}{b^2(1-b^2)}
\end{equation}
\begin{figure}[h]
\begin{center}
\center{\includegraphics[width=8cm,height=6cm,angle=0]{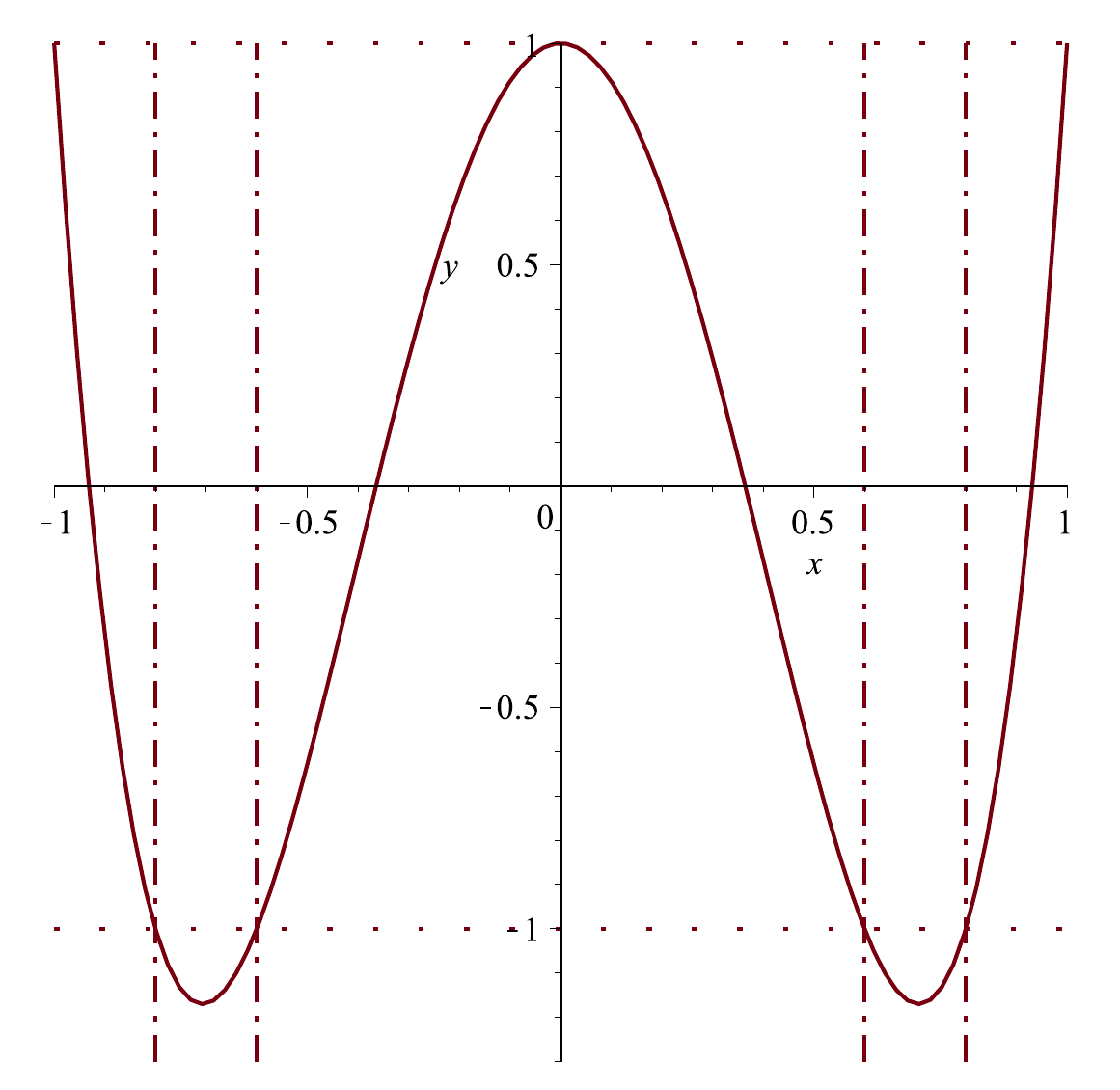}}
\end{center}
\caption{Graph of $P_4$ with $b=0.8$. The gaps in the set $E_4$ lie inside both pairs of vertical lines.}
\end{figure}
\vskip 0.2cm
\noindent
The following lemma confirms this fact.
\begin{lemma} \label{preimage}
Let $ b \in \left[\frac{1}{\sqrt{2}}, 1\right)$. The function $P_4$ defined in equation (\ref{map}) satisfies the following properties:
\begin{enumerate}
\item $P_4(\pm 1) = P_4(0) = 1$
\item $P_4(\pm b) = P_4(\pm \sqrt{1-b^2}) = -1$
\item $P_4$ is decreasing on the intervals $(-1,-b)$, $(0, \sqrt{1-b^2})$
\item $P_4$ is increasing on the intervals $[\sqrt{1-b^2}, 0]$, $[b,1]$ 
\end{enumerate} 
\end{lemma}
\begin{proof}
Items $1$ and $2$ follow directly from substitution of the points into the expression for $P_4$. The derivative of $P_4$,  denoted as $P_4^{\prime}$ from now on,  changes sign at $x = 0, \pm \frac{1}{\sqrt{2}} $. Since the intervals $(-1,-b)$ and $(0, \sqrt{1-b^2})$  do not contain any of these points for $b$ in the stated range it follows that $P_4$ is monotonic on these intervals. Combining this with $1$ and $2$ we find that $P_4$ is decreasing. $4$ follows from $3$ by symmetry. 
\end{proof}
\noindent
We now show how the polynmomial $P_4$ appears as part of a solution to a generalized Pell equation. Consider the normalized general set with two gaps
\[ [-1,\alpha_1] \cup [\beta_1, \alpha_2] \cup [\beta_2,1] \qquad -1 < \alpha_1 < \beta_1 < \alpha_2 < \beta_2 < 1 \]
The generalized Pell equation for this set is
\[ (x-\alpha_1)(x-\alpha_2)P_n^2(x) - (x^2-1)(x-\beta_1)(x-\beta_2)Q_n^2(x) = (x-\gamma_1(n))(x-\gamma_2(n) \]
where $P_n$ is a polynomial of degree $n$ and $Q_n$ is a polynomial of degree $n-1$.  As such, the Pell equation associated with the set $E_4$ is 
\begin{equation} \label{PE4}
(x-\sqrt{1-b^2})(x+b)P_n^2(x) -(x^2-1)(x-b)(x+\sqrt{1-b^2})Q_n^2(x) = (x-\gamma_1(n))(x-\gamma_2(n))
\end{equation}
For the formulation of our integral, the cases $n=2$ and $n=4$ will be of particular interest. The following two lemmas deal with these cases individually.
\begin{lemma} \label{PellP4}
The polynomial $P_4$ appearing in (\ref{map}) together with the polynomial
\[ Q_4(x) = \frac{2x}{b^2(b^2-1)}\left( x^2 + (b- \sqrt{1-b^2})x - b\sqrt{1-b^2} \right) \]
form a solution of (\ref{PE4}). 
\end{lemma}
\begin{proof}
Let 
\begin{equation} \label{gamma4}
\gamma_1(4) = -b \qquad \gamma_2(4) = \sqrt{1-b^2}.
\end{equation}
Then we can write $Q_4$ as follows :
\[ Q_4(x) =  \frac{2x}{b^2(b^2-1)} (x-\gamma_1(4))(x-\gamma_2(4)). \]
As such, the Pell equation simplifies to 
\begin{equation}
P_4^2(x) - (x^2-1)(x^2-b^2)(x^2-(1-b^2))\left(\frac{2x}{b^2(b^2-1)}\right)^2 = 1
\end{equation}
Substituting (\ref{map}) into the above equation it can be easily verified that the left hand side cancels to $1$. 
\end{proof}
\noindent
\begin{lemma}
The polynomials
\[ P_2(x) = \frac{(x-b)(x+\sqrt{1-b^2})}{b\sqrt{1-b^2}} \qquad \textrm{and} \qquad Q_2(x) = \frac{x}{b\sqrt{1-b^2}} \]
form a solution of (\ref{PE4}). 
\end{lemma}
\begin{proof}
Let
\begin{equation}
\gamma_1(2) = -\sqrt{1-b^2} \qquad \gamma_2(2) = b.
\end{equation}
With these gamma points and $P_2$ and $Q_2$ as stated above, the Pell equation simplifies to 
\begin{equation}
\frac{(x^2-(1-b^2))(x^2-b^2)}{b^2(1-b^2)} - \frac{x^2(x^2-1)}{b^2(1-b^2)} = 1
\end{equation}
The left hand side can be readily seen to cancel to $1$. 
\end{proof}
\noindent
{\it Construction of the integral}
\vskip 0.2cm
\noindent
In analogy to the previous work on the set $E_2$ we define the quantity
\begin{equation}
\begin{split}
 \cos \phi_{E_4}(x) & := \sqrt{\frac{(x-\alpha_1)(x-\alpha_2)}{(x-\gamma_1(2))(x-\gamma_2(2))}}P_2(x) \\
 & = \sqrt{\frac{(x-\sqrt{1-b^2})(x+b)}{(x+\sqrt{1-b^2})(x-b)}}P_2(x)  \\
 &  =  \sqrt{\frac{(x-\sqrt{1-b^2})(x+b)}{(x+\sqrt{1-b^2})(x-b)}}\left(\frac{(x-b)(x+\sqrt{1-b^2})}{b\sqrt{1-b^2}}\right) \qquad x \in E_4
\end{split}
\end{equation}
\begin{figure}[t] 
\begin{center}
\center{\includegraphics[width=8cm,height=6cm,angle=0]{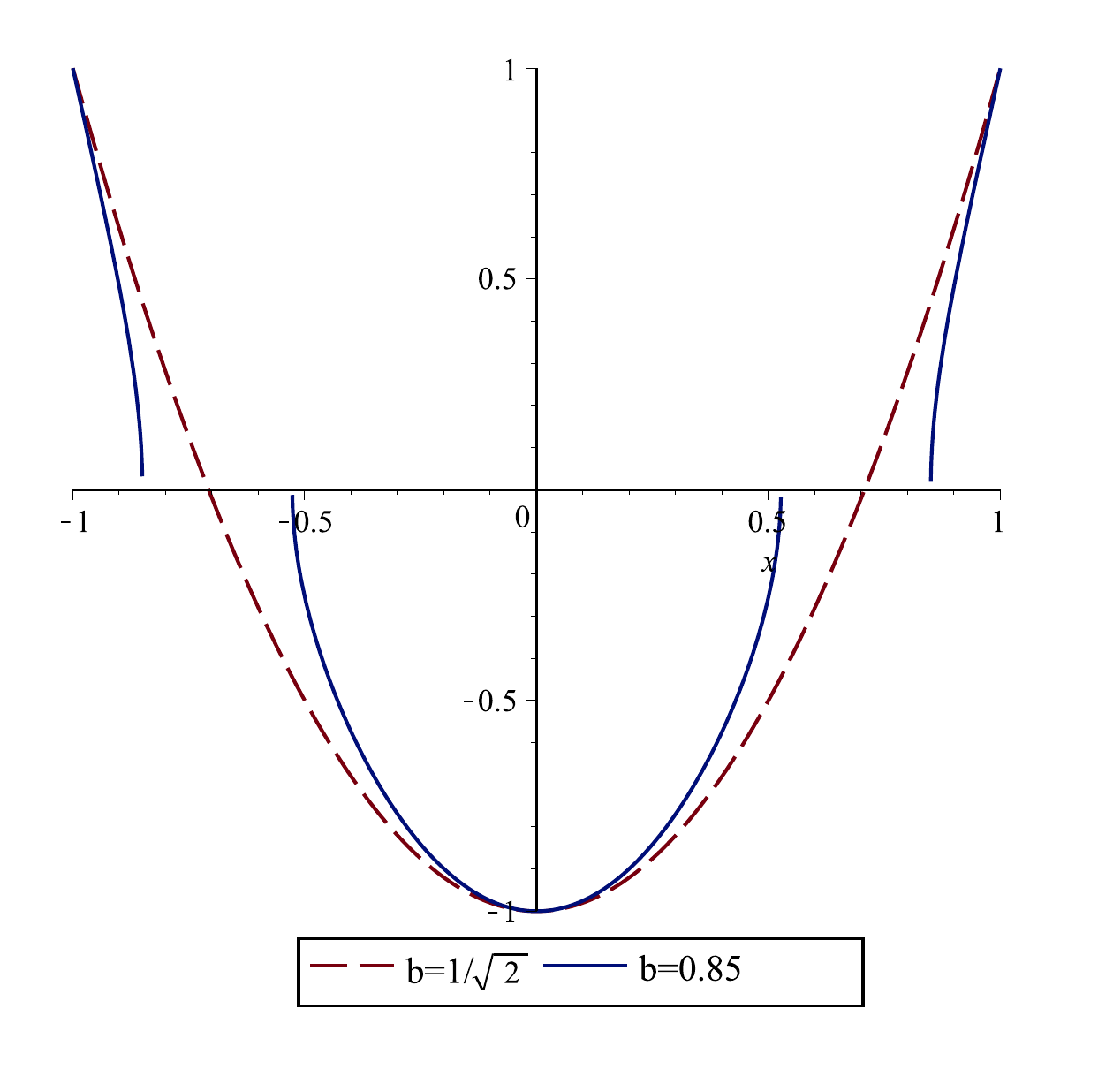}}
\end{center}
\caption{Plot of $\cos \phi_{E4}$ for two values of $b$}
\end{figure}
whose graph is shown in figure $3$. Furthermore we define the following weight function
\begin{equation}
w_{E_4}(x) = \sqrt{\frac{(x-\sqrt{1-b^2})(x+b)}{(1-x^2)(x-b)(x+\sqrt{1-b^2})}}
\end{equation}
which is the Akhiezer weight for the set $E_4$. 
In terms of these quantities, we construct the following moment integral 
\begin{equation} \label{def-int}
\int_{E_4} \left( \cos \phi_{E_4} \right)^{n}\left(1-\cos^2 \phi_{E_4} \right)^{\mu} w_{E_4} \; dx \qquad  n = 0, 1, 2,... \;\; \textrm{and} \;\;  Re(\mu) > -\frac{1}{2}
\end{equation}
where the dependence on $x$ has been suprressed for convenience of notation. Note that as $b \to \frac{1}{\sqrt{2}}$ the integral above becomes
\begin{equation}
\int_{-1}^1 \left( T_2(x)\right)^n \left(1-(T_2(x))^2 \right)^{\mu} \frac{dx}{\sqrt{1-x^2}}. 
\end{equation}
\vskip 0.2cm
\noindent
It is the purpose of this paper to study the integral (\ref{def-int}). In the next section we will show, using the theory of polynomial mappings \cite{GVA}, that the even moments are the classical ultraspherical moments and are independent of the parameter $b$. 
\vskip 0.2cm
\noindent
\begin{equation} \label{int-even}
\begin{split}
\int_{E_4} \left( \cos \phi_{E_4} \right)^{2n}\left(1-\cos^2 \phi_{E_4} \right)^{\mu} w_{E_4} \; dx = \frac{\Gamma(\frac{1}{2})\Gamma(\mu+\frac{1}{2})}{\Gamma(\mu+1)}\frac{\left(\frac{1}{2} \right)_n}{(\mu+1)_n} \\
\qquad \qquad \qquad  n = 0, 1, 2,... \;\; \textrm{and} \;\; Re(\mu) > -\frac{1}{2}
\end{split}
\end{equation}
Note that the right hand side of the above equation is equal to $\mu_{2n}$ given in (\ref{intclassiceval}). 
We will also show that the odd moments can be computed explicitly in terms of the classical hypergeometric function using the same method that was used in section $2.1$. Specifically we will show that 
\vskip 0.2cm
\noindent
\begin{equation} \label{int-odd}
\begin{split}
& \int_{E_4} \left( \cos \phi_{E_4} \right)^{2n+1}\left(1-\cos^2 \phi_{E_4} \right)^{\mu} w_{E_4} \; dx = \\
& \qquad \qquad (1-2b\sqrt{1-b^2})\frac{(1)_n}{(\mu+\frac{1}{2})_{n+1}}\;_2F_1\left(\frac{1}{2}, \mu+1; \mu + n + \frac{3}{2} ; 4b^2(1-b^2) \right) \\
& \qquad \qquad \qquad n = 0, 1, 2,... \;\; \textrm{and} \;\; Re(\mu) > -\frac{1}{2}
\end{split}
\end{equation}


\noindent
\section{Main Results}
\begin{theorem} \label{even-moments}
For $ n = 0, 1, 2,...$,  $Re(\mu) > -\frac{1}{2}$ and $ b \in \left( \frac{1}{\sqrt{2}},1 \right) $
\begin{equation}
\int_{E_4} \left( \cos \phi_{E_4} \right)^{2n}\left(1-\cos^2 \phi_{E_4} \right)^{\mu} w_{E_4} \; dx =  \frac{\Gamma(\frac{1}{2})\Gamma(\mu+\frac{1}{2})}{\Gamma(\mu+1)}\frac{\left(\frac{1}{2} \right)_n}{(\mu+1)_n}
\end{equation}
\end{theorem}
\begin{proof}
The proof is based on a substitution involving the mapping function $P_4$ given in (\ref{map}). This substitution follows the general method of polynomial mappings outlined in \cite{GVA}. It can be easily verified that 
\begin{equation}\label{evenP4}
\cos^2 \phi_{E_4} = \frac{P_4(x)+1}{2}
\end{equation}
In terms of $P_4$ the integral becomes
\begin{equation}
\int_{E_4} \left(\frac{P_4(x)+1}{2} \right)^n \left(\frac{1-P_4(x)}{2} \right)^\mu w_{E_4}  dx
\end{equation}
We know from lemma {\bf \ref{PellP4}} that $P_4$ satisfies the Pell equation,
\begin{equation}
P_4^2(x) - \frac{(x^2-1)(x-b)(x+\sqrt{1-b^2})}{(x+b)(x-\sqrt{1-b^2})}Q_4^2(x) = 1.
\end{equation}
We can rearrange the above equation to obtain
\begin{equation}
w_{E_4}(x) = \frac{|Q_4(x)|}{\sqrt{1-P_4^2(x)}}
\end{equation}
Therefore the integral can be written as
\begin{equation}
\int_{E_4} \left(\frac{P_4(x)+1}{2} \right)^n \left(\frac{1-P_4(x)}{2} \right)^\mu  \frac{|Q_4(x)|}{\sqrt{1-P_4^2(x)}} dx
\end{equation}
From lemma {\bf \ref{preimage}} we know that the set $E_4$ is the pre-image of $[-1,1]$ under the mapping $P_4$. Therefore we can write
\begin{equation}
E_4 = \cup_{i=1}^4 P^{-1}_{i,4}([-1,1])
\end{equation}
where $P_{1,4}$ is the function $P_4$ restricted to $[-1,-b]$. Similarly $P_{2,4}$ is restricted to $[-\sqrt{1-b^2},0]$, $P_{3,4}$ is restricted to $[0,\sqrt{1-b^2}]$ and $P_{4,4}$ is restricted to $[b,1]$. Using this notation the integral can now be written as a sum of four separate integrals 
\begin{equation}
\sum_{i=1}^4 \int_{P_{i,4}^{-1}([-1,1])}  \left(\frac{P_4(x)+1}{2} \right)^n \left(\frac{1-P_4(x)}{2} \right)^\mu  \frac{|Q_4(x)|}{\sqrt{1-P_4^2(x)}} dx
\end{equation}
where for each integral the function $P_4$ is monotonic. In each of these intervals we make the substitution $z = P_4(x)$ to obtain
\begin{equation}
\sum_{i=1}^4 \int_{-1}^1  \left(\frac{z+1}{2} \right)^n \left(\frac{1-z}{2} \right)^\mu  \frac{Q_4(P_{i,4}^{-1}(z))}{P_4^{\prime}(P_{i,4}^{-1}(z))}\frac{dz}{\sqrt{1-z^2}}
\end{equation}
where the absolute value signs have been removed owing to the orientation of the set $[-1,1]$. We move the summation sign inside the integral to obtain
\begin{equation}
 \int_{-1}^1  \left(\frac{z+1}{2} \right)^n \left(\frac{1-z}{2} \right)^\mu \left( \sum_{i=1}^4 \frac{Q_4(P_{i,4}^{-1}(z))}{P_4^{\prime}(P_{i,4}^{-1}(z))}\right) \frac{dz}{\sqrt{1-z^2}}
\end{equation}
We now show that the sum appearing inside the integral is equal to $1$. We start with the partial fraction decomposition of the following rational function
\begin{equation}
\frac{Q_4(s)}{P_4(s) - z} = \sum_{i=1}^4 \frac{w_i(z)}{s - P_{i,4}^{-1}(z)} \qquad z \in (-1,1) 
\end{equation}
remembering that $Q_4$ from the Pell equation is a polynomial of degree $3$ and that for $ z \in (-1,1) $ the polynomial $P_4(s) -z$ has four distinct roots. From the partial fractions decomposition it follows that  
\begin{equation}
w_i(z) =  \frac{Q_4(P_{i,4}^{-1}(z))}{P_4^{\prime}(P_{i,4}^{-1}(z))}.
\end{equation}
Since the leading coefficients of $Q_4$ and $P_4$ are the same we must have
\begin{equation} \label{sum}
\sum_{i=1}^4 w_i(z) = 1.
\end{equation}
So the integral becomes
\begin{equation}
 \int_{-1}^1  \left(\frac{z+1}{2} \right)^n \left(\frac{1-z}{2} \right)^\mu \frac{dz}{\sqrt{1-z^2}}
\end{equation}
which can be recognized as a beta integral and easily evaluates to 
\begin{equation}
 \frac{\Gamma(\frac{1}{2})\Gamma(\mu+\frac{1}{2})}{\Gamma(\mu+1)}\frac{\left(\frac{1}{2} \right)_n}{(\mu+1)_n}
\end{equation}
\end{proof}
\noindent
\begin{theorem} \label{odd-moments}
For $ n = 0, 1, 2,...$,  $Re(\mu) > -\frac{1}{2}$ and $ b \in \left( \frac{1}{\sqrt{2}},1 \right) $
\begin{equation}
\begin{split}
& \int_{E_4} \left( \cos \phi_{E_4} \right)^{2n+1}\left(1-\cos^2 \phi_{E_4} \right)^{\mu} w_{E_4} \; dx =  \\
& \qquad \qquad (1-2b\sqrt{1-b^2})\frac{(1)_n}{(\mu+\frac{1}{2})_{n+1}}\;_2F_1\left(\frac{1}{2}, \mu+1; \mu + n + \frac{3}{2} ; 4b^2(1-b^2) \right)
\end{split}
\end{equation}
\end{theorem}
\begin{proof}
We first write the integral in terms of the polynomial $P_4(x)$ as we did in the previous proof. However, since the power of $ \cos \phi_{E_4}$ is now odd we cannot make use of (\ref{evenP4}). Instead we make use of the following :
\begin{equation}
 \cos \phi_{E_4} = \left\{
\begin{array}{cc}
\sqrt{\frac{P_4+1}{2}} & x \in [-1,-b] \cup [b,1] \\
-\sqrt{\frac{P_4+1}{2}} & x \in [-\sqrt{1-b^2}, \sqrt{1-b^2} ]
\end{array}
\right.
\end{equation}
It therefore follows that 
\begin{equation}
\begin{split}
& \int_{E_4} \left( \cos \phi_{E_4} \right)^{2n+1}\left(1-\cos^2 \phi_{E_4} \right)^{\mu} w_{E_4} \; dx =  \\
& \qquad \int_{-1}^{-b} \left(\frac{P_4(x)+1}{2} \right)^{n+\frac{1}{2}} \left(\frac{1-P_4(x)}{2} \right)^{\mu} w_{E_4} \; dx  \\
& \qquad -  \int_{-\sqrt{1-b^2}}^{0} \left(\frac{P_4(x)+1}{2} \right)^{n+\frac{1}{2}} \left(\frac{1-P_4(x)}{2} \right)^{\mu} w_{E_4} \; dx \\
& \qquad -  \int^{\sqrt{1-b^2}}_{0} \left(\frac{P_4(x)+1}{2} \right)^{n+\frac{1}{2}} \left(\frac{1-P_4(x)}{2} \right)^{\mu} w_{E_4} \; dx \\
& \qquad +  \int_{b}^{1} \left(\frac{P_4(x)+1}{2} \right)^{n+\frac{1}{2}} \left(\frac{1-P_4(x)}{2} \right)^{\mu} w_{E_4} \; dx  
\end{split}
\end{equation}
In each of the four integrals appearing on the right hand side we make the change of variables $z = P_4(x)$ to obtain
\begin{equation}
\int_{-1}^1\left(\frac{z+1}{2} \right)^{n+\frac{1}{2}}\left(\frac{1-z}{2} \right)^{\mu}\left[\frac{Q_4(P_{1,4}^{-1}(z))}{P_4^{\prime}(P_{1,4}^{-1}(z))}-\frac{Q_4(P_{2,4}^{-1}(z))}{P_4^{\prime}(P_{2,4}^{-1}(z))}-\frac{Q_4(P_{3,4}^{-1}(z))}{P_4^{\prime}(P_{3,4}^{-1}(z))}+\frac{Q_4(P_{4,4}^{-1}(z))}{P_4^{\prime}(P_{4,4}^{-1}(z))} \right] \frac{dz}{\sqrt{1-z^2}}
\end{equation}
which in light of (\ref{sum}) can be written as
\begin{equation} \label{int2}
\int_{-1}^1\left(\frac{z+1}{2} \right)^{n+\frac{1}{2}}\left(\frac{1-z}{2} \right)^{\mu}\left[1-2\left(\frac{Q_4(P_{2,4}^{-1}(z))}{P_4^{\prime}(P_{2,4}^{-1}(z))}+\frac{Q_4(P_{3,4}^{-1}(z))}{P_4^{\prime}(P_{3,4}^{-1}(z))}\right) \right] \frac{dz}{\sqrt{1-z^2}}
\end{equation}
We now compute the quantity
\begin{equation}
\frac{Q_4(P_{2,4}^{-1}(z))}{P_4^{\prime}(P_{2,4}^{-1}(z))}+\frac{Q_4(P_{3,4}^{-1}(z))}{P_4^{\prime}(P_{3,4}^{-1}(z))}
\end{equation}
explicitly.
From (\ref{map}) we can easily see that 
\begin{equation}
\sqrt{2}P_{2,4}^{-1}(z) = -\sqrt{1-\sqrt{1-2(1-z)b^2(1-b^2)}}
\end{equation}
and
\begin{equation}
\sqrt{2}P_{3,4}^{-1}(z) = \sqrt{1-\sqrt{1-2(1-z)b^2(1-b^2)}}
\end{equation}
Substituting these directly into the expressions for $Q_4$ and $P_4^{\prime}$ we obtain
\begin{equation}
\frac{Q_4(P_{2,4}^{-1}(z))}{P_4^{\prime}(P_{2,4}^{-1}(z))}+\frac{Q_4(P_{3,4}^{-1}(z))}{P_4^{\prime}(P_{3,4}^{-1}(z))} = \frac{1}{2}\left(1 - \frac{1-2b\sqrt{1-b^2}}{\sqrt{1-2b^2+2b^4+[2b^2-2b^4]z}} \right).
\end{equation}
Therefore (\ref{int2}) becomes
\begin{equation}
(1-2b\sqrt{1-b^2})\int_{-1}^1\left(\frac{z+1}{2} \right)^{n+\frac{1}{2}}\left(\frac{1-z}{2} \right)^{\mu}(1-2b^2+2b^4+[2b^2-2b^4]z)^{-1/2} \frac{dz}{\sqrt{1-z^2}}
\end{equation}
By means of the subsitution $ u = \frac{z+1}{2} $ we obtain
\begin{equation}
\frac{1-2b\sqrt{1-b^2}}{2b^2-1} \int_0^1 u^{n}(1-u)^{\mu-\frac{1}{2}}\left(1-\frac{4b^2(b^2-1)}{(2b^2-1)^2}u \right)^{-1/2}du
\end{equation}
which can be written as 
\begin{equation}
\frac{1-2b\sqrt{1-b^2}}{2b^2-1}\frac{(1)_n}{(\mu+\frac{1}{2})_{n+1}}\;_2F_1\left(\frac{1}{2}, n+1; \mu + n + \frac{3}{2} ; \frac{4b^2(b^2-1)}{(2b^2-1)^2} \right).
\end{equation}
Finally, by making use of Pfaff's transformation, \cite{AAR} page 68, the above expression can be written as 
\begin{equation} \label{final-odd}
(1-2b\sqrt{1-b^2})\frac{(1)_n}{(\mu+\frac{1}{2})_{n+1}}\;_2F_1\left(\frac{1}{2}, \mu+1; \mu + n + \frac{3}{2} ; 4b^2(1-b^2) \right)
\end{equation}
\end{proof}
\noindent
From the representation for the odd moments derived above we can easily identify two limit cases. Firstly, as $b \to \frac{1}{\sqrt{2}}^+$, which corresponds to the gaps closing, it is clear that the expression (\ref{final-odd}) approaches zero. Similarly we can obtain the limiting case as the gaps approach their maximum size. By letting $b$ approach $1$ in (\ref{final-odd}) we find
\begin{equation}
\lim _{b \to 1^-}  \int_{E_4} \left( \cos \phi_{E_4} \right)^{2n+1}\left(1-\cos^2 \phi_{E_4} \right)^{\mu} w_{E_4} \; dx = \frac{(1)_n}{(\mu+\frac{1}{2})_{n+1}}.
\end{equation}
The half range moments (\ref{halfeval2}) can be written in terms of the shifted factorial as
\begin{equation}
\begin{split}
\mu_{2n}^H &= \frac{1}{2} \frac{\Gamma(\frac{1}{2})\Gamma(\mu+\frac{1}{2})}{\Gamma(\mu+1)}\frac{\left(\frac{1}{2} \right)_n}{(\mu+1)_n} \\
\mu_{2n+1}^H &= \frac{1}{2} \frac{(1)_n}{(\mu+\frac{1}{2})_{n+1}}
\end{split}
\end{equation}
and therefore, as the gaps approach their maximum size the generalized moments tend towards the half range ultraspherical moments multiplied by a factor of $2$.

\bibliographystyle{plain}

\end{document}